\documentclass[12pt]{article}
\usepackage{fullpage}
\usepackage{multirow}
\usepackage{amssymb,amsmath,amsthm,enumerate,mathrsfs,mathtools}
\usepackage{hyperref}
\usepackage{cases}
\usepackage{graphicx}
\usepackage{tikz, tikz-cd}
\usepackage{cite}
\usepackage{verbatim}
\usepackage[noblocks]{authblk}

\numberwithin{equation}{section}

\hypersetup{hidelinks} 
\theoremstyle{definition}
\newtheorem{Def}{Definition}[section]

\newtheorem{pro}[Def]{Problem}
\newtheorem{rmk}[Def]{Remark}
\newtheorem{ex}[Def]{Example}

\theoremstyle{plain}
\newtheorem{thm}[Def]{Theorem}

\newtheorem{cor}[Def]{Corollary}
\newtheorem{lem}[Def]{Lemma}

\begin{document}
\title{A New Feasibility Condition for the AT4 Family}

\author[a]{Zheng-Jiang Xia}
\author[b]{Jae-Ho Lee }
\author[c,d]{Jack H. Koolen 
	{\thanks{Corresponding author. \protect \\
	\indent \ \ 
	E-mail: xzj@mail.ustc.edu.cn (Z.J. Xia), jaeho.lee@unf.edu (J.-H. Lee), koolen@ustc.edu.cn (J.H. Koolen)}}}

\affil[a]{\small School of Finance, Anhui University of Finance \& Economics, Bengbu, 233030, China.}
\affil[b]{\small Department of Mathematics and Statistics, University of North Florida, Jacksonville, FL 32224, USA.}
\affil[c]{\small School of Mathematical Sciences, University of Science and Technology of China, Hefei, Anhui, 230026, PR China.}
\affil[d]{\small CAS Wu Wen-Tsun Key Laboratory of Mathematics, University of Science and Technology of China, 96 Jinzhai Road, Hefei, Anhui, 230026, PR China.}

\date{}

\maketitle

\begin{abstract}
Let $\Gamma$ be an antipodal distance-regular graph with diameter $4$ and eigenvalues $\theta_0>\theta_1>\theta_2>\theta_3>\theta_4$. 
Then $\Gamma$ is tight in the sense of Juri\v{s}i\'{c}, Koolen, and Terwilliger \cite{jkt2000} whenever $\Gamma$ is locally strongly regular with nontrivial eigenvalues $p:=\theta_2$ and $-q:=\theta_3$.
Assume that $\Gamma$ is tight. 
Then the intersection numbers of $\Gamma$ are expressed in terms of $p$, $q$, and $r$, where 
$r$ is the size of the antipodal classes of $\Gamma$.
We denote $\Gamma$ by  $\mathrm{AT4}(p,q,r)$ and call this an antipodal tight graph of diameter $4$ with parameters $p,q,r$.
In this paper, we give a new feasibility condition for the $\mathrm{AT4}(p,q,r)$ family.
We determine a necessary and sufficient  condition for the second subconstituent of $\mathrm{AT4}(p,q,2)$ to be an antipodal tight graph.
Using this condition, we prove that there does not exist $\mathrm{AT4}(q^3-2q,q,2)$ for $q\equiv3$ $(\mathrm{mod}~4)$.
We discuss the $\mathrm{AT4}(p,q,r)$ graphs with $r=(p+q^3)(p+q)^{-1}$.
\bigskip

\noindent
{\bf Keywords:} Distance-regular graph, Antipodal graph, Tight graph, $\mu$-graph 

\medskip
\noindent
{\bf 2020 Mathematics Subject Classification:}   05E30

\end{abstract}

\section{Introduction}
Let $\Gamma$ denote a distance-regular graph with diameter $d\geq 3$.
Let $k=\theta_0>\theta_1>\cdots>\theta_d$ denote the eigenvalues of $\Gamma$.
Juri\v{s}i\'{c} et al. \cite{jkt2000,jk2000} showed that the intersection numbers $a_1, b_1$ of $\Gamma$ satisfy the following inequality
\begin{equation}\label{eq0}
	\left(\theta_1+\frac{k}{a_1+1}\right)\left(\theta_d+\frac{k}{a_1+1}\right)\geq-\frac{ka_1b_1}{(a_1+1)^2},
\end{equation}
and defined $\Gamma$ to be \emph{tight} whenever $\Gamma$ is not bipartite, and equality holds in \eqref{eq0}.
The tight distance-regular graphs have been studied in many papers; see \cite{GT2002, j2003, jk2000, jk2002, jk2007, jk2011, jkt2000, pascasio2001} and also see \cite[Section 6.1]{vDKT}. 
A number of characterizations of the tightness property resulted; for instance, $\Gamma$ is tight if and only if $a_1\ne0$, $a_d=0$, and $\Gamma$ is $1$-homogeneous in the sense of  Nomura \cite{nomura1994}. 
In addition, $\Gamma$ is tight if and only if each local graph of $\Gamma$ is connected strongly regular, with nontrivial eigenvalues $b^+=-1-{b_1}/{(1+\theta_d)}$, and $b^-=-1-{b_1}/{(1+\theta_1)}$; cf. \cite{jkt2000}.
Juri\v{s}i\'{c} and Koolen \cite{jk2000} proved that tight distance-regular graphs with diameter three are precisely Taylor graphs, which are distance-regular graphs with intersection array $\{k,c,1;1,c,k\}$; cf \cite[Section 1.5]{BCN}. 
Moreover, by the results of \cite[Section 7]{jkt2000}, the Terwilliger algebra of a Taylor graph does not give new feasibility conditions. 
For further information on Taylor graphs and their tightness, see \cite[Section 7.6.C]{BCN}, \cite[Section 3]{jk2000}, \cite{jkt2000}, and \cite{ST1981}.

We assume that $\Gamma$ is tight with diameter four.
We further assume that $\Gamma$ is an antipodal $r$-cover.
Let $p$ and $-q$ denote the nontrivial eigenvalues of a local graph of $\Gamma$, where we assume $p>-q$.
Then all intersection numbers and eigenvalues of $\Gamma$ are expressed in terms of $p$, $q$, and $r$; cf. \cite{jk2002}.
We denote the graph $\Gamma$ by $\mathrm{AT4}(p,q,r)$ and call it an \emph{antipodal tight graph} of diameter $4$.
Juri\v{s}i\'{c} et al. \cite{jk2002, jk2000, jkt2000,jk2011,jk2007} have investigated the $\mathrm{AT4}(p,q,r)$ graphs and showed various feasibility conditions for $p$, $q$, and $r$.
Note that the family of antipodal tight graphs $\mathrm{AT4}(sq,q,q)$ are classified; cf. \cite{jk2011}.
Additionally, Koolen et al. \cite{jkt2000} showed that $\mathrm{AT4}(p,q,2)$ is pseudo-vertex-transitive by using its Terwilliger algebra.

In the present paper, we study the $\mathrm{AT}4(p,q,r)$ graphs and give a new feasibility condition for the $\mathrm{AT}4(p,q,r)$ family.
In Section \ref{prelim}, we review some preliminaries concerning the $\mathrm{AT}4(p,q,r)$ graphs.
In Section \ref{sec:NFC} we show a new feasibility condition for the $\mathrm{AT}4(p,q,r)$ graphs; see Theorem \ref{thmcon}. 
The $\mu$-graph will play an important role in this section.
Using the feasibility condition, we show that for a graph $\mathrm{AT}4(qs,q,q)$ we have $s\leq q$.
In Section \ref{sec:AT4 p,q,2} we discuss $\mathrm{AT}4(p,q,2)$ and its second subconstituent $\Delta_2$.
We give a necessary and sufficient condition for the graph $\Delta_2$ to be an antipodal tight graph; see Theorem \ref{thm4.2}.
From this result, we show the nonexistence of $\mathrm{AT}4(q^3-2q,q,2)$ when $q\equiv3~(\mathrm{mod}~4)$. 
In particular, we show that the $\mathrm{AT}4(21,3,2)$ graph does not exist.
The paper ends in Section \ref{sec: case r} with some comments on $\mathrm{AT4}(p,q,r)$ with $r=(p+q^3)(p+q)^{-1}$ and an open problem for $\mathrm{AT}4(p,q,3)$.

\section{Preliminaries}\label{prelim}
In this section, we recall some definitions and results concerning the $\mathrm{AT}4(p,q,r)$ graphs that we need later in the paper.
For more background information we refer the reader to \cite{BCN, vDKT}.
Throughout this section, let $\Gamma$ denote a simple connected graph with vertex set $V(\Gamma)$ and diameter $d$.
For $0\leq i \leq d$ and for $x\in V(\Gamma)$ we set $\Gamma_i(x)=\{y\in V(\Gamma) : \partial(x,y)=i\}$, where $\partial=\partial_{\Gamma}$ denotes the shortest path-length distance function.
For notational convenience, we define $\Gamma_{-1}(x)=\emptyset$ and $\Gamma_{d+1}(x)=\emptyset$.
We abbreviate $\Gamma(x)=\Gamma_1(x)$.
The \emph{$i$-th subconstituent} $\Delta_i(x)$ of $\Gamma$ with respect to $x\in V(\Gamma)$ is the subgraph of $\Gamma$ induced by $\Gamma_i(x)$.
We abbreviate $\Delta(x):=\Delta_1(x)$ and call this the \emph{local graph} of $\Gamma$ at $x$.
We say that $\Gamma$ is \emph{locally} $\Delta$ whenever all local graphs of $\Gamma$ are isomorphic to $\Delta$.
We say that $\Gamma$ is \emph{regular with valency $k$} (or $k$-regular) whenever  $|\Gamma(x)| = k$ for all $x\in V(\Gamma)$.
We say that $\Gamma$ is \emph{distance-regular} whenever for all integers $0\leq i \leq d$ and for all vertices $x,y \in V(\Gamma)$ with $\partial(x,y)=i$, the numbers
\begin{equation*}
	a_i = |\Gamma_i(x) \cap \Gamma(y) |, \qquad
	b_i = |\Gamma_{i+1}(x) \cap \Gamma(y) |, \qquad
	c_i = |\Gamma_{i-1}(x) \cap \Gamma(y) |
\end{equation*}
are independent of $x,y$, where we define $b_{d}:=0$ and $c_0:=0$.
Observe that $\Gamma$ is regular with valency $k=b_0$ and $a_i+b_i+c_i=k$ for $0\leq i \leq d$.
The array $\{b_0, b_1, \ldots, b_{d-1}; c_1, c_2, \ldots, c_d\}$ is called the \emph{intersection array of $\Gamma$}.

Suppose that $\Gamma$ is $k$-regular with $n$ vertices.
We say that $\Gamma$ is \emph{strongly regular} with parameters $(n,k,a,c)$ whenever each pair of adjacent vertices has the same number $a$ of common neighbors, and each pair of distinct non-adjacent vertices has the same number $c$ of common neighbors.
Note that a connected strongly regular graph is distance-regular with diameter two and parameters $(n,b_0,a_1,c_2)$.
For $x, y \in V(\Gamma)$ with $\partial(x,y)=2$, the subgraph of $\Gamma$ induced by $\Gamma(x)\cap \Gamma(y)$ is called the \emph{$\mu(x,y)$-graph} of $\Gamma$.
If the $\mu(x,y)$-graph of $\Gamma$ does not depend on the choice of $x$ and $y$, then we simply call it the \emph{$\mu$-graph} of $\Gamma$.

\begin{lem}[{cf. \cite[Theorem 3.1]{jk2000}}]\label{lem thm3.1}
Let $\Gamma$ be a distance-regular graph.
Suppose that all local graphs of $\Gamma$ are strongly regular with parameters $(n',k',a',c')$.
Then the following {\rm(i)}, {\rm(ii)} hold:
\begin{enumerate}[\normalfont(i)]
	\item $\mu$-graphs of $\Gamma$ are $c'$-regular.
	\item $c_2c'$ is even.
\end{enumerate}
\end{lem}

The graph $\Gamma$ is said to be \emph{antipodal} whenever for any vertices $x,y,z$ such that $\partial(x,y)=\partial(x,z)=d$, it follows that $\partial(y,z)=d$ or $y=z$.
The property of being at distance $d$ or zero induces an equivalence relation on $V(\Gamma)$, and the equivalence classes are called \emph{antipodal classes}.
We say that $\Gamma$ is an \emph{antipodal $r$-cover} if the equivalence classes have size $r$.

\begin{lem}[{cf. \cite[Section 4]{jk2002}}]\label{lem antipodal quo}
Let $\Gamma$ be an antipodal distance-regular graph with diameter four, $n$ vertices, valency $k$, and antipodal class size $r$.
Then the intersection array of $\Gamma$ is determined by parameters $(k,a_1,c_2,r)$, and has the following form:
\begin{equation}\label{pps3 eq(1)}
	\{b_0, b_1,b_2,b_3; c_1, c_2, c_3, c_4\}=
	\{k, k-a_1-1,(r-1)c_2,1; 1, c_2, k-a_1-1,k\}.
\end{equation}
Let $k=\theta_0>\theta_1>\theta_2>\theta_3>\theta_4$ denote the eigenvalues of $\Gamma$.
Then the parameters $a_1, c_2$ are expressed in terms of the eigenvalues and $r$:
\begin{equation}\label{pps3 eq(2)}
	a_1=\theta_1 + \theta_3, \qquad \qquad  c_2=\frac{\theta_0+\theta_2\theta_4}{r}.
\end{equation}
\end{lem}

Let $\Omega$ denote the set of triples of vertices $(x, y, z)$ of $\Gamma$ such that $\partial(x,y)=1$ and  $\partial(x,z)=\partial(y,z)=2$.
For $(x,y,z)\in \Omega$, we define the number $\alpha(x,y,z) := |\Gamma(x) \cap \Gamma(y) \cap \Gamma(z)|$, called the (\emph{triple}) \emph{intersection number} of $\Gamma$.
We say that the intersection number $\alpha$ of $\Gamma$ \emph{exists}  whenever $\alpha=\alpha(x,y,z)$ is independent of all  $(x,y,z)\in \Omega$.
If $\Gamma$ is a $1$-homogeneous graph with diameter $d\geq 2$ and $a_2\ne 0$, then the intersection number $\alpha$ of $\Gamma$ exists. 
This is because, according to the definition of $1$-homogeneity, for any two adjacent vertices $x$ and $y$ and for any vertex $z \in \Gamma_2(x)\cap \Gamma_2(y)$, the scalar $\alpha=|\Gamma(z)\cap \Gamma(x)\cap \Gamma(y)|$ is constant; see \cite[Lemma 11.5]{jkt2000}.
A strongly regular graph with $a_2\ne 0$, that is locally strongly regular is $1$-homogeneous   if and only if $\alpha$ exists; cf. \cite{jkt2000}.

We now recall an antipodal tight graph $\mathrm{AT4}(p,q,r)$.
In the following three lemmas, we review some properties concerning $\mathrm{AT4}(p,q,r)$ from \cite{jk2002}, which will be used later.

\begin{lem}[cf. \cite{jk2002}]\label{pps1}
Let $\Gamma$ denote an antipodal tight graph $\mathrm{AT4}(p,q,r)$. 
Then the following {\rm(i)}--{\rm(iv)} hold.
\begin{enumerate}[\normalfont(i)]
	\item The graph $\Gamma$ has nontrivial eigenvalues $\theta_1>\theta_2>\theta_3>\theta_4$, where
	\begin{equation}\label{nontv eig}
	\theta_1 = pq+p+q, \quad \theta_2=p, \quad \theta_3=-q, \quad \theta_4=-q^2,
	\end{equation}
	and its intersection array is
	\begin{equation*}
	\begin{split}
	& \left\{q(pq+p+q),(q^2-1)(p+1),\frac{(r-1)q(p+q)}{r},1; \right. \\
	& \left. \qquad 1,\frac{q(p+q)}{r},(q^2-1)(p+1),q(pq+p+q)\right\}.
	\end{split}
	\end{equation*}

	\item The local graph of $\Gamma$ at each vertex is strongly regular with parameters $(n',k',a',c')=(q(pq+p+q), p(q+1),2p-q,p)$, and its spectrum is given by
	\begin{equation}\label{local srg spec}
	\begin{pmatrix}
	p(q+1) & p & -q \\
	1 & \ell_1 & \ell_2  \\
	\end{pmatrix},
	\end{equation}
	where 
	\begin{equation}\label{ell1,2}
	\ell_1 :=\frac{(q^2-1)(pq+p+q)}{p+q} \quad \text{and} \quad \ell_2:=\frac{pq(p+1)(q+1)}{p+q}.
	\end{equation}

	\item The graph $\Gamma$ is $1$-homogeneous. In particular, $\alpha=(p+q)/r$.

	\item The parameters $p,q,r$ are integers such that $p\geq1,q\geq2,r\geq2$ and	
	\begin{enumerate}[\normalfont(1)]
		\item $pq(p+q)/r$ is even, $r(p+1)\leq q(p+q)$, and $r|p+q$,
		\item $p\geq q-2$, with equality if and only if the Krein parameter $q_{4,4}^4=0$,
		\item $(p+q)|q^2(q^2-1)$ and $(p+q^2)|q^2(q^2-1)(q^2+q-1)(q+2)$.
	\end{enumerate}
\end{enumerate}
\end{lem}

\noindent
We remark that by Lemma \ref{pps1}(i) one readily finds the intersection numbers $\{a_i\}^4_{i=0}$ of $\Gamma$:
\begin{equation}\label{AT4 int num ai}
	a_0=a_4=0, \qquad a_1=a_3=p(q+1), \qquad a_2=pq^2.
\end{equation}

\begin{lem}[cf. {\cite[Theorem 4.3]{jk2007}}]\label{lem mu-graph}
Let $\Gamma$ be an antipodal tight graph $\mathrm{AT4}(p,q,r)$ with $p>1$.
Then its $\mu$-graphs are complete multipartite if and only if there exists an integer $s$ such that $(p,q,r)=(qs,q,q)$.
\end{lem}

\begin{lem}[cf. {\cite[Corollary 4.5]{jk2007}}]\label{pps2}
Let $\Gamma$ denote an antipodal tight graph $\mathrm{AT4}(p,q,r)$. Then exactly one of the following statements holds.

\begin{enumerate}[\normalfont(i)]
	\item $\Gamma$ is the unique $\mathrm{AT4}(1,2,3)$ graph (and $\alpha=1$), i.e., the Conway-Smith graph.
	\item $\Gamma$ is an $\mathrm{AT4}(q-2,q,q-1)$ graph (and $\alpha=1$).
	\item $\Gamma$ is an $\mathrm{AT4}(qs,q,q)$ graph, where $s$ is an integer (and $\alpha=s+1$).
	\item $(p+q)(2q+1)\geq3r(p+2)$ and $\alpha\geq3$, in particular, $r\leq q-1$.
\end{enumerate}
\end{lem}

Lastly, we recall the spectral excess theorem \cite{van2008}.
Recall the graph $\Gamma$ with vertex set $V(\Gamma)$ and diameter $d$.
Denote the spectrum of $\Gamma$ by
$\mathrm{Spec}(\Gamma)=\{{\lambda_0}^{m_0}, {\lambda_1}^{m_1}, \ldots, {\lambda_d}^{m_d}\}$.
Let $\mathcal{P}$ denote the vector space of polynomials of degree at most $d$.
With reference to $\mathrm{Spec}(\Gamma)$ define an inner product on $\mathcal{P}$ by
\begin{equation}\label{inn prod}
	\langle p,q \rangle = \frac{1}{n} \sum^d_{i=0} m_i p(\lambda_i) q(\lambda_i).
\end{equation}
With respect to \eqref{inn prod}, there exists a unique system of orthogonal polynomials $\{p_i\}^d_{i=0}$ such that $p_i$ has degree $i$ and $\langle p_i, p_i\rangle = p_i(\lambda_0)$ for $0\leq i \leq d$.
\begin{lem}[cf. {\cite[Theorem 1]{van2008}}]\label{spectral excess thm}
Let $\Gamma$ be a connected $k$-regular graph on $n$ vertices with diameter $d$.
Let $\{p_i\}^d_{i=0}$ be the orthogonal polynomials corresponding to $\Gamma$.
If $k_d(x)$ is the number of vertices at distance $d$ from a vertex $x$ in $\Gamma$, then
\begin{equation}\label{ineq xp.ex.thm}
	n-p_d(k) \leq  n \left[ \sum_{x\in V(\Gamma)} \frac{1}{n-k_d(x)} \right]^{-1},
\end{equation}
with equality if and only if $\Gamma$ is distance-regular.
\end{lem}


\section{A new feasibility condition}\label{sec:NFC}
In this section, we introduce a new feasibility condition for the $\mathrm{AT4}(p,q,r)$ family.
We use the following notation.
Let $\Gamma$ denote an antipodal tight graph $\mathrm{AT4}(p,q,r)$. 
Fix a vertex $x$ in $\Gamma$.
Choose a vertex $y$ in $\Gamma$ with $\partial(x,y)=2$.
Consider the antipodal class containing $y$, denoted by $\{y=y_1, y_2, \ldots, y_r\}$.
We define the subgraph $H$ of $\Gamma$ as the union of the $\mu$-graphs $\Gamma(x)\cap\Gamma(y_i)$ of $\Gamma$ for all $1\leq i \leq r$:
\begin{equation}\label{graph H}
	H:=\bigcup^r_{i=1} \Gamma(x)\cap\Gamma(y_i).
\end{equation}
Observe that $H$ is $p$-regular and $|V(H)|=q(p+q)$.

\begin{lem}\label{eig H subgraph}
Let $H$ be the graph as in \eqref{graph H}. Then the following {\rm(i)}--{\rm(iii)} hold.
\begin{enumerate}[\normalfont(i)]
	\item $H$ has $p$ as an eigenvalue of multiplicity $r$.
	\item $H$ has $-q$ as an eigenvalue of multiplicity at least $pq(1+p+q-q^2)/(p+q)$.
	\item $H$ has at least three distinct eigenvalues.
\end{enumerate}
\end{lem}
\begin{proof}
(i) Since $H$ has $r$ connected components and each component is $p$-regular, the result follows. \\
(ii) Consider the local graph $\Delta=\Delta(x)$ of $\Gamma$.
By Lemma \ref{pps1}(ii), $\Delta$ is strongly regular with parameters $(n',k',a',c')$ and the spectrum \eqref{local srg spec}. 
Denote the eigenvalues of $\Delta$ by $\delta_1\geq \delta_2 \geq \cdots \geq \delta_{n'}$, where $n'=q(pq+p+q)$.
By \eqref{local srg spec}, we find that $\delta_i=-q$ for all $2+\ell_1\leq i \leq n'$, where $\ell_1$ is from \eqref{ell1,2}.
Denote the eigenvalues of $H$ by $\varepsilon_1 \geq \varepsilon_2 \geq \cdots \geq \varepsilon_m$, where $m=|V(H)|=q(p+q)$.
Since $H$ is a subgraph of $\Delta$, by interlacing we have $\delta_i \geq \varepsilon_i \geq \delta_{n'-m+i}$ for $1\leq i \leq m$.
Evaluate these inequalities at $i=2+\ell_1$ and $i=m$, respectively, and combine the two results to get
$$
	-q=\delta_{2+\ell_1} \geq \varepsilon_{2+\ell_1} \geq \varepsilon_m \geq  \delta_{n'}=-q.
$$
From this, it follows that $\varepsilon_j=-q$ for all $2+\ell_1 \leq j \leq m$.
Thus, the multiplicity of $-q$ is at least $q(p+q)-1-\ell_1$.
Simplify this quantity to get the desired result.\\
(iii) Since the $\mu$-graph is a subgraph of $H$, it suffices to show that the $\mu$-graph has diameter at least $2$. 
If the $\mu$-graph has diameter $1$, then it must be the complete graph $K_{p+1}$.
Since $|V(H)|=q(p+q)=r(p+1)$ and by Lemma \ref{lem mu-graph}, we have $q=1$, a contradiction.
Therefore, the $\mu$-graph has diameter at least $2$, as desired.
\end{proof}

\begin{thm}\label{thmcon}
Let $\Gamma$ be an antipodal tight graph $\mathrm{AT4}(p,q,r)$. Then
\begin{equation}\label{eqcon}
	r\leq\frac{p+q^3}{p+q}.
\end{equation}
If the equality holds, then the $\mu$-graph of $\Gamma$ is strongly regular with parameters 
\begin{equation}\label{mu srg para}
	\left(\frac{q(p+q)}{r},p,(q-1)(q-2)+\frac{2(p-1)}{q+1},\frac{p+q^3}{q+1}\right).
\end{equation}
\end{thm}
\begin{proof}
Let $H$ be the subgraph of $\Gamma$ as in \eqref{graph H}.
By Lemma \ref{eig H subgraph}(i), (ii), $H$ has eigenvalues $p$ with multiplicity $r$ and $-q$ with multiplicity at least $pq(1+p+q-q^2)/(p+q)$, denoted by $\sigma$.
By Lemma \ref{eig H subgraph}(iii), $H$ has (possibly repeated) eigenvalues distinct from $p$ and $-q$, denoted by $\lambda_1,\lambda_2,\ldots,\lambda_\tau$ for some $\tau$.
Since $|V(H)|=q(p+q)$, we have
\begin{equation}\label{pf thm3.2 eq1}
	\tau=q(p+q)-r-\sigma.
\end{equation}
Let $B$ denote the adjacency matrix of $H$.
Then one readily checks that $\mathrm{tr}(B)=0$ and $\mathrm{tr}(B^2)=pq(p+q)$.
Using these equations and linear algebra, we have
\begin{equation}\label{pf thm3.2 eq2}
	\sum_{i=1}^{\tau} \lambda_i  = \sigma q - rp, \qquad
	\sum_{i=1}^{\tau} \lambda_i^2  = pq(p+q) - rp^2 - \sigma q^2.
\end{equation}
By the Cauchy-Schwartz inequality, we have 
\begin{equation}\label{CSineq}
	\left(\sum_{i=1}^{\tau} \frac{\lambda_i}{\tau} \right)^2
	\leq
	\sum_{i=1}^{\tau} \frac{\lambda_i^2}{\tau}. 
\end{equation}
Evaluate \eqref{CSineq} using \eqref{pf thm3.2 eq1} and \eqref{pf thm3.2 eq2} and simplify the result to get
\begin{equation}\label{eq1}
	\frac{q^2-1}{p+q}+\frac{q(p+q)-r(p+1)}{pq-r(p+q)+q^3}\geq1.
\end{equation}
Verify $pq-r(p+q)+q^3>0$ by considering each case of Lemma \ref{pps2}.
Using this inequality, solve \eqref{eq1} for $r$ and simplify the result to obtain \eqref{eqcon}.

For the second assertion, the equality in \eqref{eqcon} holds if and only if the equality in \eqref{CSineq} holds if and only if there exists $\nu\in \mathbb{R}$ such that $\lambda_i = \nu \tau^{-1}$ for all $i$ $(1\leq i \leq \tau)$.
Thus, if the equality holds in \eqref{eqcon}, then we find that the $\mu$-graph has three distinct eigenvalues: $p$, $-q$, and $\nu \tau^{-1}$.
Therefore, the $\mu$-graph is strongly regular.
The parameters \eqref{mu srg para} follow routinely.
\end{proof}

\begin{ex}
Consider the $\mathrm{AT}4(351, 9, 3)$ graph.
One checks that $r=(p+q^3)(p+q)^{-1}$ with $p=351,q=9,r=3$.
By Theorem \ref{thmcon}, the $\mu$-graph of $\mathrm{AT}4(351, 9, 3)$ is strongly regular with parameters $(1080, 351, 126, 108)$; see Table \ref{table1}.
\end{ex}

\begin{rmk}
The converse of the second statement in Theorem \ref{thmcon} does not hold in general.
For example, the $\mu$-graph of $\mathrm{AT}4(2,2,2)$ is $K_{2,2}$, which is a strongly regular graph. 
However, equality in \eqref{eqcon} does not hold since $r =2 \ne 5/2 = (p+q^3)(p+q)^{-1}$.
\end{rmk}

\begin{rmk}\label{spec mu graph}
In the proof of Theorem \ref{thmcon}, we saw that the $\mu$-graph has the eigenvalue $\lambda:=\nu \tau^{-1}$ when equality holds in \eqref{eqcon}.
By the equation on the left in \eqref{pf thm3.2 eq2} it follows that $\lambda = (\sigma q - rp)\tau^{-1}$.
Evaluate this equation using \eqref{pf thm3.2 eq1} to get $\lambda=(p-q^2)(1+q)^{-1}$. 
Observe that $p>\lambda>-q$.
The spectrum of the $\mu$-graph is 
\begin{equation*}
	\left(
	\begin{array}{ccc}
	p & \frac{p-q^2}{1+q} & -q \\[5pt]
	1 & \frac{p(q-1)(q+1)^2}{p+q^3} & \frac{pq(1+p+q-q^2)}{p+q^3}
	\end{array}
	\right).
\end{equation*}
\end{rmk}

\begin{cor}\label{corrq}
For an antipodal tight graph $\mathrm{AT4}(qs,q,q)$, we have $s\leq q$.
\end{cor}
\begin{proof}
The result follows from \eqref{eqcon}.
\end{proof}

We have a comment on a bound for $p$.
By the first expression in Lemma \ref{pps1}(iv)(3), we find an upper bound for $p$, namely, $p\leq q^4-q^2-q$. 
In the following, by using Theorem \ref{thmcon} we obtain a better bound for $p$.

\begin{cor}\label{corp}
For an antipodal tight graph $\mathrm{AT4}(p,q,r)$, we have $p\leq q^3-2q$.
\end{cor}
\begin{proof}
By Theorem \ref{thmcon} and since $r\geq 2$, we have $2\leq (p+q^3)(p+q)^{-1}$.
The result follows.
\end{proof}

\section{The graph $\mathrm{AT4}(q^3-2q,q,2)$}\label{sec:AT4 p,q,2}
In this section we discuss the antipodal tight graphs $\mathrm{AT4}(p,q,2)$ and their second subconstituent graphs.
We find the spectrum of this second subconstituent of $\mathrm{AT4}(p,q,2)$.
We then give a necessary and sufficient condition for this second subconstituent to be an antipodal tight graph with diameter four.
Let $\Gamma$ denote an antipodal tight graph $\mathrm{AT4}(p,q,2)$.
For $0\leq i \leq 4$, consider the $i$-th subconstituent $\Delta_i=\Delta_i(x)$ of $\Gamma$ with respect to a vertex $x$ in $\Gamma$.
Note that $\Delta_1$ is isomorphic to $\Delta_3$. 
For $0\leq i \leq 4$ let $k_i$ denote the cardinality of the vertex set of $\Delta_i$.
One readily finds that
\begin{align}
	&  k_0=k_4 =1, \qquad k_1=k_3 =q(pq+p+q), \\[5pt]
	& k_2 = \frac{2(pq+p+q)(q^2-1)(p+1)}{p+q}. \label{k2}
\end{align}
By Lemma \ref{pps1}(ii), we see that $\Delta_i$ ($i=1,3$) is strongly regular and its spectrum is given by \eqref{local srg spec}.
We now discuss the spectrum of $\Delta_2$ in detail.
To this end, we begin with the following lemma that will be used shortly.

\begin{lem}\label{lem tr(B^n)}
For a given vertex $x$ of $\Gamma$, let $B$ denote the adjacency matrix of the second subconstituent $\Delta_2(x)$ of $\Gamma$.
Then 
\begin{enumerate}[\normalfont(i)]
	\item $\mathrm{tr}(B)=0$,
	\item $\mathrm{tr}(B^2)=\dfrac{2pq^2(pq+p+q)(q^2-1)(p+1)}{p+q}$,
	\item $\mathrm{tr}(B^3)=\dfrac{2pq^3(pq+p+q)(q^2-1)(p^2-1)}{p+q}$.
\end{enumerate}
\end{lem}
\begin{proof}
(i) Clear.\\
(ii) Observe that $\mathrm{tr}(B^2)$ is the total number of closed $2$-walks in $\Delta_2$. This number is equal to $a_2 k_2$. Evaluate this using \eqref{AT4 int num ai} and \eqref{k2}.\\
(iii) Observe that $\mathrm{tr}(B^3)$ is the total number of directed $3$-cycles in $\Delta_2$.
This number is equal to $a_2k_2h$, where $h$ is the number of triangles containing one given edge in $\Delta_2$. 
By construction, we find $h=a_1-\alpha r$, where $r=2$ and $\alpha=(p+q)/2$ by Lemma \ref{pps1}(iii).
Evaluate $a_2k_2(a_1-2\alpha)$ using \eqref{AT4 int num ai} and \eqref{k2}.
\end{proof}

\begin{lem}\label{lemeig}
Let $\Gamma$ denote an antipodal tight graph $\mathrm{AT4}(p,q,2)$. For each vertex $x\in V(\Gamma)$, the spectrum of the second subconstituent $\Delta_2(x)$ is
\begin{equation}\label{eq spec AT4(p,q,2)}
\begin{pmatrix}
 pq^2& pq & p+q-q^2 & p & -q & -q^2 \\[5pt]
1 & \dfrac{(q^2-1)(pq+p+q)}{p+q} & \dfrac{pq(p+1)(q+1)}{p+q}& m_1 & m_2 & m_3 \\
\end{pmatrix},
\end{equation}
where
\begin{align}
&m_1=\frac{-q(p+1)(p-q^3+2q)(pq+p+q)}{(p + q)(p+q^2 )},\label{scalar m1}\\
 &m_2= \frac{p(q^2-1)(pq+p+q)}{p + q},\label{scalar m2}\\
 &m_3= \frac{p(p-q+2)(q^2-1)(pq+p+q)}{(p+q)(p+q^2)}.\label{scalar m3}
\end{align}
\end{lem}
\begin{proof}
By the proof of \cite[Lemma 8.5]{klt2022}, $\Delta_2=\Delta_2(x)$ has at most seven distinct eigenvalues, denoted by
$$
	pq^2, \quad pq, \quad p+q-q^2, \quad  pq+p+q, \quad p,\quad  -q, \quad -q^2.
$$
Also, by the proof of \cite[Lemma 8.5]{klt2022} the multiplicity of $pq$ (resp. $p+q-q^2$) is equal to the multiplicity of the eigenvalue $p$ (resp. $-q$) of $\Delta(x)$.
By these comments and Lemma \ref{pps1}(ii), we may denote the spectrum of $\Delta_2$ by
\begin{equation*}
\begin{pmatrix}
 pq^2& pq & p+q-q^2 & \theta_1 & \theta_2 & \theta_3 & \theta_4 \\[5pt]
1 & \ell_1 & \ell_2 & m_0 & m_1 & m_2 & m_3 
\end{pmatrix},
\end{equation*}
where $\{\theta_i\}^4_{i=1}$ are from \eqref{nontv eig} and $\ell_1, \ell_2$ are from \eqref{ell1,2}.

We find the multiplicities $m_i$ $(0\leq i \leq 3)$.
Let $B$ denote the adjacency matrix of $\Delta_2$.
By linear algebra, we have
\begin{equation}\label{pf trace B^n}
	\mathrm{tr}(B^j) = (pq^2)^{j} + (pq)^{j}\ell_1 +  (p+q-q^2)^{j}\ell_2 + \sum^3_{i=0} m_i\theta_{i+1}^j,
\end{equation}
for a nonnegative integer $j$.
For each $j=0,1,2,3$, evaluate \eqref{pf trace B^n} using \eqref{k2} and Lemma \ref{lem tr(B^n)} to get a system of four linear equations in four variables $m_0, m_1, m_2, m_3$.
Solve this system of equations to get
\begin{align*}
&m_0=0,\\
&m_1=\frac{-q(p+1)(p-q^3+2q)(pq+p+q)}{(p + q)(p+q^2 )},\\
 &m_2= \frac{p(q^2-1)(pq+p+q)}{p + q},\\
 &m_3= \frac{p(p-q+2)(q^2-1)(pq+p+q)}{(p+q)(p+q^2)}.
\end{align*}
The result follows.
\end{proof}

\begin{rmk}\label{rmk multi}
We verify that the expressions of $\{m_i\}^3_{i=1}$ in \eqref{scalar m1}--\eqref{scalar m3} are integers.
Recall the integers $\ell_1$, $\ell_2$ from \eqref{ell1,2}.
Observe that the expression on the right in \eqref{scalar m2} is equal to $p\ell_1$, which is a positive integer.
Note that the expression on the right in \eqref{scalar m3} can be expressed as
\begin{equation}\label{alt m3}
	(q^2-1)(pq+p-q^3-3q^2+q+2)-\frac{2q^2(q^2-1)}{p+q}+\frac{q^2(q^2-1)(q^2+q-1)(q+2)}{p+q^2}.
\end{equation}
By Lemma \ref{pps1}(iv)(3), the expression \eqref{alt m3} is an integer.
By these comments, it follows that the expression on the right in \eqref{scalar m1} becomes an integer.
\end{rmk}

We now give a necessary and sufficient condition for the second subconstituent of $\Gamma$ to be antipodal tight.

\begin{thm}\label{thm4.2}
Let $\Gamma$ denote an antipodal tight graph $\mathrm{AT4}(p,q,2)$. 
For each vertex $x\in V(\Gamma)$, the second subconstituent $\Delta_2(x)$ of $\Gamma$ is an antipodal distance-regular graph with diameter four if and only if $p=q^3-2q$. 
Moreover, $\Delta_2(x)$ is an $\mathrm{AT4}(q^3-q^2-q,q,2)$ graph when $p=q^3-2q$.
\end{thm}
\begin{proof}
Abbreviate $\Delta_2=\Delta_2(x)$.
Suppose that $\Delta_2$ is an antipodal distance-regular graph with diameter four.
Then $\Delta_2$ has precisely five distinct eigenvalues, which implies that one of $m_i$ $(i=1,2,3)$ in \eqref{eq spec AT4(p,q,2)} must be zero. 
We claim $m_1=0$.
As we saw in Remark \ref{rmk multi}, $m_2>0$.
If $m_3=0$, by \eqref{scalar m3} and since $p\geq 1$ and $q\geq 2$ we have $p=q-2$; this is a contradiction to \cite[Theorem 2]{g13}.
Therefore, we need to have $m_1=0$, as claimed.
Then, by \eqref{scalar m1} we have $p=q^3-2q$.

Conversely, suppose that $p=q^3-2q$.
Then, from \eqref{scalar m1}--\eqref{scalar m3} we find that $m_1=0$ and each $m_i$ $(i=2,3)$ is nonzero. 
By this and Lemma \ref{lemeig}, $\Delta_2$ has precisely five distinct eigenvalues.
It follows that $\Delta_2$ has diameter at most $4$.
Next, recall the vertex set $\Gamma_2=\Gamma_2(x)$ of $\Delta_2$ and pick a vertex $v\in \Gamma_2$.
Let $k_4(v)$ denote the number of vertices in $\Gamma_2$ at distance $4$ from $v$.
We claim that $k_4(v)\geq 1$.
Consider the antipodal vertex $u \in \Gamma_2$ of $v$.
Then $\partial(u,v)=4$ in $\Gamma$, which implies that the distance between $u$ and $v$ in $\Delta_2$ is at least $4$.
However, since the diameter of $\Delta_2$ is at most $4$, the distance between $u$ and $v$ in $\Delta_2$ must be $4$.
From this, we find that $\Delta_2$ has diameter $4$ and $k_4(v)\geq 1$, as claimed.

We apply Lemma \ref{spectral excess thm} to $\Delta_2$, and then apply the above claim to the right-hand side of \eqref{ineq xp.ex.thm} to get
\begin{equation}\label{eq sp.ex.thm}
	n-p_4(\kappa)\leq n \left[ \sum_{v\in \Gamma_2} \frac{1}{n-k_4(v)} \right]^{-1} \leq n-1,
\end{equation}
where $n=|\Gamma_2|$ and $\kappa$ is the valency of $\Delta_2$.
Calculate $p_4(\kappa)$ using \cite[Section 6]{van2008} with \eqref{eq spec AT4(p,q,2)}.
Then we find $p_4(\kappa)=1$.
By this, the equality in \eqref{eq sp.ex.thm} holds.
Therefore, by Lemma \ref{spectral excess thm} we have that $\Delta_2$ is distance-regular.
Moreover, $k_4(v)=1$ for all $v\in \Gamma_2$, and hence $\Delta_2$ is an antipodal $2$-cover with diameter $4$.

Next, we show that $\Delta_2$ is an antipodal tight graph $\mathrm{AT4}(q^3-q^2-q,q,2)$ when $p=q^3-2q$.
Recall the spectrum of $\Delta_2$ from \eqref{eq spec AT4(p,q,2)}.
Since $m_1=0$, for notational convenience we denote the eigenvalues of $\Delta_2$ by
\begin{equation}\label{eq eig Delta_2}
	\theta_0:= pq^2 	\ > \
	\theta_1:= pq 		\ > \
	\theta_2:= p+q-q^2 	\ > \
	\theta_3:= -q 		\ > \
	\theta_4:= -q^2.
\end{equation}
Since $\Delta_2$ is antipodal distance-regular with diameter $4$, express $a_1$ and $c_2$ in \eqref{pps3 eq(2)} in terms of $q$ using \eqref{eq eig Delta_2} and $p=q^3-2q$.
Using \eqref{pps3 eq(1)} together with parameters $k=\theta_0$, $a_1$, $c_2$, $r=2$, we find the intersection array of $\Delta_2$:
\begin{equation}\label{int num D_2}
\begin{split}
	& \left\{q^3(q^2-2),(q-1)^3(q+1)^2,\frac{q^3(q-1)}{2},1; \right.\\
	& \qquad \qquad \left. 1,\frac{q^3(q-1)}{2},(q-1)^3(q+1)^2,q^3(q^2-2)\right\}.
\end{split}
\end{equation}
Compare \eqref{int num D_2} with the intersection array in \cite[Theorem 5.4(ii)]{jk2002}.
The result follows.
\end{proof}

\begin{cor}\label{cor nonexi}
If $\mathrm{AT4}(q^3-q^2-q,q,2)$ does not exist, neither does $\mathrm{AT4}(q^3-2q,q,2)$.
\end{cor}
\begin{proof}
It directly follows from Theorem \ref{thm4.2}.
\end{proof}

\begin{cor}\label{cornonext}
Neither $\mathrm{AT4}(q^3-q^2-q,q,2)$ nor $\mathrm{AT4}(q^3-2q,q,2)$ exists when $q\equiv3~(\mathrm{mod}~4)$.
\end{cor}
\begin{proof}
By Corollary \ref{cor nonexi}, it suffices to show that $\mathrm{AT4}(q^3-q^2-q,q,2)$ does not exist for $q\equiv3~(\mathrm{mod}~4)$.
Suppose that there exists $\Gamma=\mathrm{AT4}(q^3-q^2-q,q,2)$, where $q\equiv3~(\mathrm{mod}~4)$.
By Lemma \ref{pps2}(ii), a local graph of $\Gamma$ is strongly regular with parameters $(n',k',a',c')$, where we notice that $c'= q^3-q^2-q$.
By \eqref{int num D_2}, the intersection number $c_2$ of $\Gamma$  is $q^3(q-1)/2$.
Thus, $c_2c'=q^3(q-1)(q^3-q^2-q)/2$.
This quantity is odd since $q\equiv3~(\mathrm{mod}~4)$, which contradicts Lemma \ref{lem thm3.1}(ii).
Therefore, such a graph $\Gamma$ does not exist.
\end{proof}

We have some comments.
Juri\v{s}i\'{c} presented a table that lists some known examples  and open cases of the $\mathrm{AT}4$ family; see \cite[Table 2]{j2003}.
We improve this table by using the results of the present paper and the result in \cite{g13} as follows.
By \cite[Theorem 2]{g13} B2, B6, and B11 in \cite[Table $2$]{j2003} are ruled out.
Since A10 and B8 in \cite[Table $2$]{j2003} satisfy the equality in Theorem \ref{thmcon}, their $\mu$-graphs are strongly regular; however, B8 should be ruled out by Corollary \ref{cornonext}.
Note that the smallest eigenvalue of the $\mu$-graph of $\mathrm{AT}4(p,q,r)$ is $-q$ by Lemma \ref{eig H subgraph}. 
By this note, we find that the $\mu$-graph of B4 in \cite[Table $2$]{j2003} cannot be $K_{9,9}$, and the $\mu$-graph of B5 in \cite[Table $2$]{j2003} cannot be $2\cdot K_{8,8}$.
Based on these comments above, \cite[Table $2$]{j2003} has been updated; see the new version, Table \ref{table1}.

\begin{table}
\caption{The AT4 family, $\alpha=(p+q)/r$, $c_2=q\alpha$.} \label{table1}
\medskip
(a) \textrm{Known examples, where ``!'' indicates the uniqueness of the corresponding graph.}
\begin{center}
\begin{tabular}{l l r r r r r r l}
\hline 
\# & graph & $k$ & $p$ & $q$ & $r$ & $\alpha$ & $c_2$ & $\mu$-graph \\ 
\hline \hline
A1 & ! Conway-Smith & 10 & 1 & 2  &  3    & 1 & 2 & $K_2$ \\
A2 & ! J(8,4)  & 16  & 2  & 2  &  2    & 2 & 4 & $K_{2,2}$ \\
A3   & ! halved 8-cube  &  28 & 4  & 2  &   2   & 3 & 6 & $K_{3\times 2}$ \\
A4  & ! $3.O_6^-(3)$  & 45  & 3  & 3  &  3    & 2 & 6 & $K_{3,3}$ \\
A5  & ! Soicher1  & 56  & 2  & 4  & 3     &2  & 8 & $2\cdot K_{2,2}$ \\
A6 & ! $3.O_7(3)$  & 117  & 9  & 3  & 3     & 4 & 12 & $K_{4\times 3}$ \\
A7 & Meixner1  &  176 & 8  & 4  &  2    & 6 & 24 & $2\cdot K_{3\times 4}$ \\
A8  & ! Meixner2  & 176  & 8  & 4  &  4    & 3 & 12 & $K_{3\times 4}$ \\
A9 &Soicher2   & 416  & 20  & 4  &  3    & 8 & 32 & $\overline{K_2}$-ext. of $\frac{1}{2}Q_5$ \\
A10 & $3.Fi_{24}^-$  & 31671  & 351  & 9  &  3    & 120 & 1080 & ${\rm SRG}(1080,351,126,108)$ \\ \hline
\end{tabular}
\end{center}

(b) \textrm{Remaining open cases of small members of the $\mathrm{AT4}$ family on at most $4096$ vertices (with valency $k \leq 416$) and some ideas for their $\mu$-graphs (whose valency is $p$).}
\begin{center}
\begin{tabular}{l l r r r r r r l}
\hline
\# & graph \hspace{1.8cm}{ } & \hspace{0.6cm}{ } $k$ &  \hspace{0.1cm}{ } $p$ & $q$  & $r$ &  \hspace{0.2cm}{ } $\alpha$ & \hspace{0.4cm}{ } $c_2$ & $\mu$-graph \hspace{2.5cm}{ } \\ \hline \hline
B1 &  & 96 & 4 & 4  &  2    & 4 & 16 & $2\cdot K_{4,4}$ \\
B2 & does not exist & 115  & 3  & 5  &  2    & 4 & 20 & $2\cdot\rm Petersen$ \\
B3   &  &  115 & 3  & 5  &   4   & 2 & 10 &Petersen \\
B4  &   & 117  & 9  & 3  &  2    & 6 & 18 & unknown; not $K_{9,9}$ \\
B5  &  & 176  & 8  & 4  & 3     &4  & 16 & unknown; not $2\cdot K_{8,8}$\\
B6 & does not exist  & 204  & 4  & 6  & 2     & 5 & 30 & $5\cdot K_{3\times 2}$ \\
B7 &   &  204 & 4  & 6  &  5    & 2 & 12 & $2\cdot K_{3\times 2}$ \\
B8& does not exist &261&21  &3  &2    &12 &36  & ${\rm SRG}(36,21,12,12)$\\
B9 &   & 288  & 6  & 6  &  2    & 6 & 36 & $3\cdot K_{6,6}$ \\
B10 &   &  288 & 6  & 6  &  3    & 4 & 24 & $2\cdot K_{6,6}$ \\
B11  & does not exist  & 329  & 5  & 7  &  2    & 6 & 42 & $7\cdot K_{6}$ \\
B12 &   & 336  & 16  & 4  &  2    & 10 & 40 & $2\cdot K_{5\times 4}$ \\
B13 &   & 416  & 20  & 4  &  2   & 12 & 48  & $2\cdot K_{6\times 4}$ \\ \hline
\end{tabular}
\end{center}
\end{table}


\medskip
We finish this section with a comment.

\begin{thm}
Let $\Gamma$ denote an antipodal tight graph $\mathrm{AT4}(p,3,r)$.
Then $\Gamma$ must be one of the following {\rm(i)} -- {\rm(iii)}:
\begin{enumerate}[\normalfont(i)]
	\item $\mathrm{AT4}(3,3,3)$, that is $\Gamma$ is the $3.O^-_6(3)$ graph; {\rm cf. \cite[Section 13.2C]{BCN}}.
	\item $\mathrm{AT4}(9,3,3)$, that is $\Gamma$ is the $3.O_7(3)$ graph; {\rm cf. \cite[Section 13.2D]{BCN}}.
	\item $\mathrm{AT4}(9,3,2)$.
\end{enumerate}
\end{thm}
\begin{proof}
By Lemma \ref{pps2}, we have $r=2$ or $r=3$.
If $r=3$, $\Gamma$ is either $\mathrm{AT4}(3,3,3)$ or $\mathrm{AT4}(9,3,3)$ by \cite[Theorem 5.1]{jk2011}.
The uniqueness of $3.O_7(3)$ refers to \cite{JK manuscript}.
If $r=2$, then $p\leq 21$ by Corollary \ref{corp}.
All possible parameters are listed in Table \ref{table1}, in which $\mathrm{AT4}(9,3,2)$ is the unique open case.
\end{proof}

\section{Case $r=(p+q^3)(p+q)^{-1}$}\label{sec: case r}
In Section \ref{sec:NFC}, we gave a new feasibility condition \eqref{eqcon} for the $\mathrm{AT4}(p,q,r)$ family. 
In this section, we give some comments on the graphs $\mathrm{AT4}(p,q,r)$ when equality holds in \eqref{eqcon}, that is, $r=(p+q^3)(p+q)^{-1}$. 
Let $\Gamma$ denote an antipodal tight graph $\mathrm{AT4}(p,q,r)$ with $r=(p+q^3)(p+q)^{-1}$. 
Recall $r \geq 2$.
If $r=2$, then we have $p=q^3-2q$. 
It turns out that $\Gamma$ is $\mathrm{AT4}(q^3-2q,q,2)$, which has been treated in Section \ref{sec:AT4 p,q,2}.
Assume that $r>2$.
Then we have the following feasibility condition for $\Gamma$.
\begin{lem}\label{lem1}
With the above notation, we have $(q+r)|r(r-2)(r-1)^2(r^2-r-1)$.
\end{lem}
\begin{proof}
Solve the equation $r=(p+q^3)(p+q)^{-1}$ for $p$ to get $p=(q^3-rq)(r-1)^{-1}$.
By Lemma \ref{pps1}(iv)(3), we know that 
\begin{equation}\label{eq:lem1 pf}
	(p+q^2)|q^2(q^2-1)(q^2+q-1)(q+2).
\end{equation} 
Substitute $p=(q^3-rq)(r-1)^{-1}$ in \eqref{eq:lem1 pf} and simplify the result to get 
$$
	(q+r)|(r-1)q(q+1)(q+2)(q^2+q-1).
$$
Set $h(q):=(r-1)q(q+1)(q+2)(q^2+q-1)$.
Then there exist polynomials $f(q)$ and $g(r)$ such that $h(q)=f(q)(q+r)+g(r)$.
By this comment, we find that $(q+r)|h(q)$ if and only if $(q+r)|g(r)$.
Put $q=-r$ in $h(q)=f(q)(q+r)+g(r)$ to get $g(r)$.
The result follows.
\end{proof}

\begin{cor}
For $r>2$, the set of feasible parameters $\{(p,q) : r=(p+q^3)(p+q)^{-1}\}$ is finite.
\end{cor}
\begin{proof}
By Lemma \ref{lem1},  $q+r$ divides $r(r-2)(r-1)^2(r^2-r-1)$. 
Since the set of such $q$ is finite, the result follows.
\end{proof}

\begin{ex}
Consider an antipodal tight graph $\mathrm{AT4}(p,q,3)$.
If $(p+q^3)(p+q)^{-1}=3$, by Lemma \ref{lem1} we have $(q+3)|60$.
The possible values for $q$ are 
$$
	2, 3, 7, 9, 12, 17, 27, 57.
$$
Table \ref{table1} shows that the $\mathrm{AT}4(p,q,3)$ exists for $q = 2$, $3$, and $9$, which are $\mathrm{AT4}(1,2,3)$, $\mathrm{AT4}(9,3,3)$, and $\mathrm{AT4}(351,9,3)$, respectively.
Note that the existence for other values of $q$ is unknown.
\end{ex}

We finish the paper with the open problem.

\begin{pro}
Classify all $\mathrm{AT4}(p,q,3)$ graphs with $(p+q^3)(p+q)^{-1}=3$.
\end{pro}

\section*{Acknowledgements}

Zheng-Jiang Xia is supported by Anhui Natural Science Foundation (2108085MA01), NSFC (No.11601002), and Key Projects in Natural Science Research of Anhui Provincial Department of Education (Nos. KJ2020A0015 and KJ2018A0438). J.H. Koolen is partially supported by the National Natural Science Foundation of China (No. 12071454), Anhui Initiative in Quantum Information Technologies (No. AHY150000) and the National Key R and D Program of China (No. 2020YFA0713100).

\end{document}